\documentclass[leqno,12pt]{article}
\usepackage{latexsym}
\usepackage{amsmath}
\usepackage{amssymb}
\usepackage{enumerate}

\usepackage{theorem}
\newtheorem{theorem}{Theorem}[section]

\newtheorem{lemma}[theorem]{Lemma}

\theorembodyfont{\rmfamily}
\newtheorem{proof}{\textmd{\textit{Proof.}}}

\newtheorem{remark}[theorem]{Remark}
\newtheorem{example}[theorem]{Example}
\newtheorem{definition}[theorem]{Definition}

\makeatletter

\@addtoreset{equation}{section}
\makeatother

\newcommand{\qedd}{\hfill \Box}
\newcommand{\ve}{\varepsilon}
\newcommand{\lra}{\longrightarrow}
\newcommand{\ora}{\overrightarrow}
\newcommand{\ola}{\overleftarrow}
\newcommand{\wt}{\widetilde}

\newcommand{\ol}{\overline}

\newcommand{\N}{\ensuremath{\mathbb{N}}}
\newcommand{\R}{\ensuremath{\mathbb{R}}}
\newcommand{\Sph}{\ensuremath{\mathbb{S}}}

\newcommand{\cG}{\ensuremath{\mathcal{G}}}

\newcommand{\cN}{\ensuremath{\mathcal{N}}}
\newcommand{\cT}{\ensuremath{\mathcal{T}}}
\newcommand{\cO}{\ensuremath{\mathcal{O}}}

\newcommand{\cR}{\ensuremath{\mathcal{R}}}

\def\Cut{\mathop{\mathrm{Cut}}\nolimits}
\def\Conj{\mathop{\mathrm{Conj}}\nolimits}
\def\CC{\mathop{\mathrm{CC}}\nolimits}

\title{A Toponogov type triangle comparison theorem\\
in Finsler geometry
}

\author{Kei KONDO\ $\cdot$\
 Shin-ichi OHTA\footnote{Supported in part by Grant-in-Aid for Young Scientists (B) 23740048}
 \ $\cdot$\ Minoru TANAKA}
\date{\today}
\begin{document}
\maketitle

\begin{abstract}
The aim of this article is to establish a Toponogov type triangle comparison theorem
for Finsler manifolds, in the manner of radial curvature geometry.
We consider the situation that the radial flag curvature is bounded
below by the radial curvature function of a non-compact surface of revolution,
the edge opposite to the base point is contained in a Berwald-like region,
and that the Finsler metric is convex enough in the radial directions in that region.
\end{abstract}

\section{Introduction}\label{sec1}

Toponogov's triangle comparison theorem (TCT)
is one of the milestones in global Riemannian geometry.
It asserts that the angles of a geodesic triangle in a complete Riemannian manifold
of non-negative sectional curvature are not smaller than the corresponding angles
of a triangle with the same side lengths in $\R^2$ (\cite{To1}).
TCT was first used to generalize Cohn-Vossen's splitting
theorem (\cite[Satz~5]{CV2}) to higher dimensions (\cite{To2}),
and then has played vital roles in the study of the relationship
between curvature and topology of Riemannian manifolds
(see \cite{GM}, \cite{CG}, \cite{GS}, \cite{G} etc.). 
Fruitful non-trivial consequences can be produced from 
simple geometric techniques using geodesic triangles (as in Euclidean geometry) via TCT.\par 
There are several generalizations of TCT.
One can weaken the assumption by restricting the curvature bound
only to the radial direction from a fixed base point $p$
(then $p$ must be a vertex of the triangle).
It is also possible to consider a model surface other than $\R^2$,
even the one changing the sign of its Gaussian curvature
(see \cite{IMS}, the first and the third authors' \cite{KT2} and \cite{KT3}).
Some applications of TCT are also generalized. For example, 
the diameter sphere theorem of Grove and Shiohama \cite{GS} was extended by the first and the second 
authors in that radial curvature sense (see \cite{KO}).\par 
In \cite{O}, the second author studied a related comparison theorem in Finsler geometry.
Estimated in \cite[Theorem~4.2]{O} was the concavity of the squared distance function
depending on the lower flag curvature bound, the lower tangent curvature bound,
and on the $2$-uniform smoothness of tangent spaces.
The flag curvature is a generalization of the sectional curvature,
while the latter two quantities do not appear in the Riemannian setting.
Because of these non-Riemannian quantities, it is in general impossible
to extend TCT to Finsler manifolds.
It is especially difficult to give the angle a sense.
Then the aim of this article is {\em to find a reasonable non-Riemannian situation
where TCT makes sense}.
It will turn out that a certain triangle in a certain (non-Riemannian)
Finsler manifold can satisfy TCT in a reasonable way.
We remark that, on comparison geometry and geometric analysis related to the Ricci curvature, 
there is a successful theory by the second author and Sturm
for general Finsler manifolds (see \cite{Oint}, \cite{OShf}, \cite{OSbw}, \cite{Ospl}).
\bigskip

In order to state our main theorem, let us introduce several notions.
Let $(M, F, p)$ denote a pair of a forward complete, connected $C^\infty$-Finsler manifold $(M,F)$
and a base point $p \in M$.
Because its distance function $d$ is not symmetric in general, we introduce
\[
d_{\rm{m}} (x, y) := \max\{d(x, y), d(y, x)\}.
\]
It is clear that $d_{\rm{m}}$ is a distance function of $M$.
We can define the `angles' with respect to $d_{\rm{m}}$ as follows.
Let $c :[0,a] \lra M$ be a unit speed minimal geodesic segment
(i.e., $F(\dot{c}) \equiv 1$) with $p \not\in c([0,a])$. 
The \emph{forward} and the \emph{backward angles}
$\ora{\angle}(pc(s)c(a)),\ola{\angle}(pc(s)c(0)) \in [0,\pi]$ at $c(s)$ are defined via
\begin{align}
\cos \ora{\angle}\big( pc(s)c(a) \big) &:= -\lim_{h \downarrow 0} 
 \frac{d(p,c(s+h)) -d(p, c(s))}{d_{\rm{m}} (c(s), c(s + h))} \quad \text{for $s \in [0,a)$},
 \label{forwardcos1} \\
\cos \ola{\angle}\big( pc(s)c(0) \big) &:= \lim_{h \downarrow 0} 
 \frac{d(p, c(s)) - d (p, c(s-h))}{d_{\rm{m}} (c(s -h), c(s))} \quad \text{for $s \in (0, a]$}.
 \label{forwardcos2}
\end{align}
These limits indeed exist in $[-1,1]$ (see Lemma~\ref{2011_05_27_lem1},
we use $d_{\rm{m}}$ rather than $d$ for ensuring that they live in $[-1,1]$).
Given three distinct points $p, x, y \in M$, we will denote by 
\[ \triangle (\ora{px}, \ora{py}) := (p, x, y; \gamma, \sigma, c) \]
the \emph{forward triangle} consisting of unit speed minimal geodesic segments
$\gamma$ emanating from $p$ to $x$, $\sigma$ from $p$ to $y$, and $c$ from $x$ to $y$.
Define the interior angles $\ora{\angle}x, \ola{\angle}y$ at the vertices $x$, $y$ by 
\[ \ora{\angle}x := \ora{\angle}(pxy), \qquad
 \ola{\angle}y := \ola{\angle}(pyx). \]

For a local coordinate $(x^i)^{n}_{i=1}$ of an open subset $\cO$ of $M$,
let $(x^i, v^j)_{i,j=1}^{n}$ be the coordinate of the tangent bundle $T\cO$ over $\cO$ such that 
\[ v:= \sum_{j = 1}^{n} v^j \frac{\partial}{\partial x^j}\Big|_{x},
 \qquad \ x \in \cO. \]
For each $v \in T_xM \setminus \{0\}$, the positive-definite $n \times n$ matrix 
\[
\big( g_{ij} (v) \big)_{i,j= 1}^{n}:= 
\left(
\frac{1}{2} \frac{\partial^2 (F^2)}{\partial v^i \partial v^j}(v)
\right)_{i, j = 1}^{n}
\]
provides us the Riemannian structure $g_{v}$ of $T_x M$ by 
\begin{equation}\label{rieman_str}
g_{v} \left(
\sum_{i = 1}^{n} a^i \frac{\partial}{\partial x^i}\bigg|_{x}, 
\sum_{j = 1}^{n} b^j \frac{\partial}{\partial x^j}\bigg|_{x}
\right) 
:=
\sum_{i,j = 1}^{n} g_{ij} (v) a^ib^j. 
\end{equation}
This is a Riemannian approximation (up to the second order)
of $F$ in the direction $v$.
For two linearly independent vectors $v, w \in T_{x} M \setminus \{0\}$,
the {\em flag curvature} is defined by 
\begin{equation}\label{eq:flag}
K_M (v, w) := \frac{g_{v} (R^{v} (w, v)v, w)}{g_{v} (v, v) g_{v} (w, w) - g_{v} (v, w)^2},
\end{equation}
where $R^{v}$ denotes the curvature tensor induced from the Chern connection (see \cite[\S 3.9]{BCS}). 
We remark that $K_M (v, w)$ depends not only on the \emph{flag} $\{sv + tw\,|\, s, t \in \R\}$,
but also on the \emph{flag pole} $\{sv\,|\, s> 0\}$.

Given $v,w \in T_xM \setminus \{0\}$, define the \emph{tangent curvature} by 
\begin{equation}\label{def_T_curv}
\cT_M(v, w) := g_X\big( D^Y_Y Y(x) - D^X_Y Y(x), X(x) \big), 
\end{equation}
where $X,Y$ are extensions of $v,w$,
and $D_{v}^{w}X(x)$ denotes the covariant derivative of $X$ by $v$ with reference vector $w$. 
Independence of $\cT_M(v,w)$ from the choices of $X,Y$ is easily checked.
Note that $\cT_M \equiv 0$ if and only if $M$ is of \emph{Berwald type}
(see \cite[Propositions~7.2.2, 10.1.1]{Sh}).
In Berwald spaces, for any $x,y \in M$, 
the tangent spaces $(T_xM, F|_{T_xM})$ and $(T_yM, F|_{T_yM})$
are mutually linearly isometric (cf.~\cite[Chapter~10]{BCS}).
In this sense, $\cT_M$ measures the variety of tangent Minkowski spaces.
\bigskip

Let $(\wt{M},\tilde{p})$ denote a complete $2$-dimensional Riemannian manifold 
homeomorphic to $\R^{2}$ with a base point $\tilde{p} \in \wt{M}$
such that its Riemannian metric $d\tilde{s}^2$ is expressed 
in terms of geodesic polar coordinate around $\tilde{p}$ as 
\[
d\tilde{s}^2 = dt^2 + f(t)^2d \theta^2, \qquad 
(t,\theta) \in (0,\infty) \times \Sph_{\tilde{p}}^1,
\]
where $f : (0, \infty) \lra \R$ is a positive smooth function 
which is extensible to a smooth odd function around $0$.
Define the {\em radial curvature function} $G: [0,\infty) \lra \R$
such that $G(t)$ is the Gaussian curvature at $\tilde{\gamma}(t)$,
where $\tilde{\gamma}(t):[0,\infty) \lra \wt{M}$ is any meridian (emanating from $\tilde{p}$).
Note that $f$ satisfies the differential equation 
$f''+Gf=0$ with initial conditions $f(0) = 0$ and $f'(0) = 1$. 
We call $(\wt{M}, \tilde{p})$ a {\em von Mangoldt surface}
if $G$ is non-increasing on $[0, \infty)$.
Paraboloids and $2$-sheeted hyperboloids are typical examples of a von Mangoldt surface.
An atypical example of a von Mangoldt surface is the following.

\begin{example}{\rm (\cite[Example 1.2]{KT1})}\label{2012_03_10_exa1.2}
Set $f (t) := e^{- t^{2}} \tanh t$ on $[0, \infty)$.
Then the surface of revolution $(\wt{M}, \tilde{p})$ 
with $d\tilde{s}^2 = dt^2 +  f(t)^2d \theta^2$ is of von Mangoldt type, 
and its radial curvature function 
\[
G (t) = -\,\frac{f''(t)}{f(t)} = \frac{8t}{\sinh 2t} + \frac{2}{\cosh^{2}t} -4 t^{2} + 2
\]
changes the sign on $[0, \infty)$, indeed,
$\lim_{t \downarrow 0}G (t) = 8$ and $\lim_{t \to \infty}G (t) = - \infty$.
Note that the total curvature is equal to $2 \pi$, since $\lim_{t \to \infty} f'(t) = 0$. 
\end{example}

We say that a Finsler manifold $(M, F, p)$ has the
{\em radial flag curvature bounded below by that of $(\wt{M}, \tilde{p})$} if, 
along every unit speed minimal geodesic $\gamma: [0,l) \lra M$ 
emanating from $p$, we have
\[
K_{M} \big(\dot{\gamma}(t), w \big) \ge G (t)
\]
for all $t \in [0, l)$ and $w \in T_{\gamma(t)}M$ 
linearly independent to $\dot{\gamma}(t)$. 
Given a forward triangle $\triangle (\ora{px}, \ora{py}) \subset M$,
a geodesic triangle $\triangle (\tilde{p}\tilde{x} \tilde{y}) \subset \wt{M}$ is called
its \emph{comparison triangle} if
\[ \tilde{d}(\tilde{p}, \tilde{x}) = d(p, x), \qquad 
\tilde{d}(\tilde{p},\tilde{y}) = d(p, y), \qquad 
\tilde{d}(\tilde{x},\tilde{y}) = L_{\rm{m}}(c) \]
hold, where $\tilde{d}$ denotes the distance function of $d\tilde{s}^2$ and we set
\[
L_{\rm{m}}(c):= \int^{d(x,\,y)}_0 \max\{ F(\dot{c}),F(-\dot{c}) \} \,ds.
\]
Note that a comparison triangle is unique up to an isometry of $\wt{M}$. 
Now we state our main theorem.

\begin{theorem}{{\bf (TCT)}}\label{TCT}
Let $(M, F, p)$ be a forward complete, connected $C^{\infty}$-Finsler manifold
whose radial flag curvature is bounded below by that of a von Mangoldt surface 
$(\wt{M}, \tilde{p})$ satisfying $f'(\rho) =0$ and $G(\rho)\ne0$ for a unique $\rho \in (0, \infty)$. 
Let $\triangle (\ora{px}, \ora{py}) = (p, x, y; \gamma, \sigma, c) \subset M$ be a forward triangle 
satisfying that, for some open neighborhood $\cN(c)$ of $c$,
\begin{enumerate}[{\rm (1)}]
\item
$c([0,d(x,y)]) \subset M \setminus \ol{B^+_{\rho} (p)}$,
\item 
$g_v(w,w) \ge F(w)^2$ for all $z \in \cN(c)$, $v \in \cG_p(z)$ and $w \in T_zM$,
\item
$\cT_{M}(v,w) = 0$ for all $z \in \cN(c)$, $v \in \cG_p(z)$ and $w \in T_zM$,
and the reverse curve $\bar{c}(s):=c(d(x,y)-s)$ $(s \in [0,d(x,y)])$ of $c$ is geodesic.
\end{enumerate}
If such $\triangle (\ora{px}, \ora{py})$ admits a comparison triangle $\triangle (\tilde{p}\tilde{x} \tilde{y})$ in $\wt{M}$,
then we have $\ora{\angle} x \ge \angle \tilde{x}$ and $\ola{\angle} y \ge \angle \tilde{y}$.
\end{theorem}

See \eqref{eq:Gp} for the definition of $\cG_p(z)$,
and $B^+_{\rho}(p)$ denotes the forward open ball (see \eqref{eq:ball}). 
The hypothesis (2) means that $F^2|_{T_zM}$ is convex enough in the direction
near $v$ (cf.\ \cite[\S 4.3]{O}), and (3) holds true when $F$ is of Berwald type on $\cN(c)$
for instance.
We  employ $L_{\rm{m}}(c)$ because $\bar{c}$ may not be minimal.
If $\bar{c}$ is minimal, then $L_{\rm{m}}(c)=d_{\rm{m}}(x,y)$ simply holds.
We remark that $c$ is not necessarily geodesic with respect to the Riemannian structure $g_v$.
Indeed, (3) implies only $g_v(D^v_{\dot{c}} \dot{c},v)=0$,
and $D^v_{\dot{c}} \dot{c}$ does not coincide with the corresponding 
covariant derivative with respect to $g_v$ in general.

It is not difficult to construct non-Riemannian examples satisfying the conditions in the theorem.
Let us start with a Riemannian manifold $(M,g)$ satisfying the curvature condition.
We modify (the unit spheres of) $g$ on $M \setminus B^+_{\rho}(p)$,
outside a neighborhood of $\bigcup_{z \in M \setminus B^+_{\rho}(p)} \cG_p(z)$,
in such a way that (2) holds.
Since this modification does not affect $g_v$ for $v \in \bigcup_{z \in M \setminus B^+_{\rho}(p)} \cG_p(z)$,
the resulting non-Riemannian metric still satisfies (3).
Another example is a Finsler manifold $(M,F)$ satisfying the curvature condition
such that $F$ is Riemannian on $M \setminus B^+_{\rho}(p)$.

\begin{remark}
There are several applications of Theorem~\ref{TCT}, e.g., we proved 
the finiteness of topological type and a diffeomorphism theorem to Euclidean spaces (\cite{KOT}), and the first author generalized the diameter sphere theorem of Grove and Shiohama 
by modifying Theorem~\ref{TCT} (\cite{K}).
\end{remark}

The organization of this article is as follows.
In Section~\ref{sec2}, we verify the validity of the forward and the backward angles. 
Then we apply the techniques developed by the first and the third authors (\cite{KT2}, \cite{KT3})
in Sections~\ref{sec3}--\ref{sec6}.
Section~\ref{sec3} is devoted to the key estimate (Lemma~\ref{lem2.6}) 
on the lengths of geodesic variations. From this, we readily derives weak TCT 
(i.e., TCT with respect to slightly worse model surfaces) for thin triangles outside 
the cut or conjugate locus of the base point in Section~\ref{sec4}.
In Section~\ref{sec5}, we prove the double triangle lemma for model surfaces of revolution,
which enables us to glue two thin triangles.
We finally show TCT by gluing thin triangles and improving the model surface in Section~\ref{sec6}.

\section{Angles}\label{sec2}

Throughout the article, $(M,F,p)$ denotes a forward complete, connected
$C^\infty$-Finsler manifold with a base point $p \in M$,
and $d$ denotes its distance function.
We refer to \cite{BCS} for the basics of Finsler geometry.
The forward completeness guarantees that any two points in $M$ can be joined by
a minimal geodesic segment (by the Hopf-Rinow theorem, \cite[Theorem~6.6.1]{BCS}).
We will not assume the \emph{reversibility} $F(-v)=F(v)$.
Thus $d(x,y) \neq d(y,x)$ can happen and the reverse curve of a geodesic
is not necessarily geodesic.

For each $x \in M \setminus \{ p\}$, we set 
\begin{equation}\label{eq:Gp}
\mathcal{G}_p (x) := \{ \dot{\gamma}(l) \in T_{x}M \,|\,
 \text{$\gamma$ is a minimal geodesic segment from $p$ to $x$} \},
\end{equation}
where $\gamma:[0,l] \lra M$ with $l=d(p, x)$. Recall \eqref{rieman_str} for the definition of $g_v$. 

\begin{theorem}{\rm (\cite[Proposition 2.1]{TS})}\label{2011_05_27_thm1}
Take $x \in M \setminus \{p\}$ and let $\{x_{i}\}_{i=1}^{\infty} \subset M \setminus \{p,x\}$
be a sequence converging to $x$.
For sufficiently large $i$, we set 
\[ w_i := \frac{1}{F(\exp_{x}^{-1} (x_i))}  \exp_{x}^{-1} (x_i) \ \in T_xM, \]
where $\exp_{x}^{-1} (x_i)$ means the initial velocity 
of the unique minimal geodesic segment emanating from $x$ to $x_i$.
If $w := \lim_{i \to \infty} w_i$ exists, then we have
\[ \lim_{i \to \infty} \frac{d (p, x_i) -d(p, x)}{d(x, x_i)}
 =\min \{ g_{v} (v,w) \,|\, v \in \mathcal{G}_{p}(x)\}. \]
\end{theorem}

\begin{lemma}\label{2011_05_27_lem1}
Let $c :[0, a] \lra M$ be a unit speed minimal geodesic segment satisfying $p \not\in c([0,a])$. 
Then the forward and the backward angles as in \eqref{forwardcos1} and \eqref{forwardcos2} 
are well-defined.
\end{lemma}

\begin{proof}
We need to take a little care on the non-reversibility of $F$.
Fix $s \in [0,a)$.
If $F(\dot{c}(s))=1 \ge F(-\dot{c}(s))$, then we have
\[ \lim_{h \downarrow 0} \frac{d(c(s),c(s + h))}{d_{\rm{m}}(c(s),c(s + h))} =1 \]
and hence Theorem~\ref{2011_05_27_thm1} yields
\begin{align*}
\lim_{h \downarrow 0} \frac{d (p, c(s + h)) -d(p, c(s))}{d_{\rm{m}} (c(s), c(s + h))} 
&= \lim_{h \downarrow 0} \frac{d (p, c(s + h)) -d(p, c(s))}{d(c(s), c(s + h))} \\
&= \min \big\{ g_{v} \big( v,\dot{c}(s) \big) \,|\, v \in \mathcal{G}_{p}\big( c(s) \big) \big\}.
\end{align*}
Note that, for every $v \in \mathcal{G}_p(c(s))$,
\begin{equation}\label{eq:angle}
g_v\big( v,\dot{c}(s) \big) \le F(v) F\big( \dot{c}(s) \big) =1, \qquad
g_v\big( v,\dot{c}(s) \big) \ge -F(v) F\!\left( -\dot{c}(s) \big) \right. \ge -1.
\end{equation}
If $F(-\dot{c}(s))>1$, then
\[ \lim_{h \downarrow 0} \frac{d(c(s),c(s + h))}{d_{\rm{m}}(c(s),c(s + h))}
 =\frac{1}{F(-\dot{c}(s))} \]
and Theorem~\ref{2011_05_27_thm1} imply
\[ \lim_{h \downarrow 0} \frac{d (p, c(s + h)) -d(p, c(s))}{d_{\rm{m}} (c(s), c(s + h))} 
 =\frac{1}{F(-\dot{c}(s))}
 \min \big\{ g_{v} \big( v,\dot{c}(s) \big) \,\big|\, v \in \mathcal{G}_{p}\big( c(s) \big) \big\}. \]
Since $|g_v(v,\dot{c}(s))| \le F(-\dot{c}(s))$,
\eqref{forwardcos1} exists in $[-1,1]$ in both cases.

The same argument shows that \eqref{forwardcos2} is well-defined.
Precisely, for $s \in (0,a]$, we have
\[ \lim_{h \downarrow 0} \frac{d(p,c(s-h)) -d(p,c(s))}{d(c(s),c(s-h))} 
 =\min \left\{ g_{v} \left( v,\frac{-\dot{c}(s)}{F(-\dot{c}(s))} \right) \,\bigg|\, v \in \mathcal{G}_{p}\big( c(s) \big) \right\} \]
and hence, by letting $\lambda:=\max\{1,F(-\dot{c}(s))\}$,
\[ \lim_{h \downarrow 0} \frac{d (p, c(s)) -d(p, c(s-h))}{d_{\rm{m}} (c(s-h), c(s))} 
 =\lambda^{-1} \max \big\{ g_{v} \big( v,\dot{c}(s) \big) \,\big|\, v \in \mathcal{G}_{p}\big( c(s) \big) \big\}. \]
$\qedd$
\end{proof}

\section{Key lemma on lengths of geodesic variations}\label{sec3}

This section is devoted to the key comparison estimates between the lengths
of geodesic variations in $M$ and a model surface $\wt{M}$.
For $x \in M$ and $r>0$, define the \emph{forward} and the \emph{backward open balls} by
\begin{equation}\label{eq:ball}
B^+_r(x):=\{ y \in M \,|\, d(x,y)<r \}, \qquad B^-_r(x):=\{ y \in M \,|\, d(y,x)<r \},
\end{equation}
respectively.
We also set $B^{\pm}_r(x):=B^+_r(x) \cap B^-_r(x)$.

\subsection{Preliminaries for geodesic variations}\label{sec3.1}

Let $\Cut (p)$ be the cut locus of $p$. 
Take a point $q \in M \setminus (\Cut (p) \cup \{p\})$ and small $r>0$ such that
\[ B^-_{2r}(q) \cap (\Cut (p) \cup \{p\} )= \emptyset \]
and that $B_r^{\pm}(q)$ is geodesically convex
(i.e., any minimal geodesic joining two points in $B_r^{\pm}(q)$ is contained in $B_r^{\pm}(q)$).
Given a unit speed minimal geodesic $c:(-\ve,\ve) \lra B_r^{\pm}(q)$,
we consider the $C^{\infty}$-variation 
\[
\varphi (t, s) := \exp_{p}\left( \frac{t}{l} \exp_p^{-1}\big( c(s) \big) \right),
 \qquad (t,s) \in [0,l] \times (-\ve,\ve),
\]
where $l=d(p,c(0))$.
Since $x:=c(0) \not\in \Cut(p)$, there is a unique minimal geodesic segment
$\gamma:[0,l] \lra M$ emanating from $p$ to $x$.
By setting 
\[ J(t) := \frac{\partial \varphi}{\partial s} (t, 0), \]
we get the Jacobi field $J$ along $\gamma$ with $J(0) = 0$ and $J (l) = \dot{c}(0)$.

\begin{remark}\label{rem_2012_01_21_3.1}
Since $\gamma$ is unique, it follows from Lemma~\ref{2011_05_27_lem1} that 
\[ -\cos \ora{\angle}\big( pxc(\ve) \big) =\cos \ola{\angle}\big( pxc(-\ve) \big)
 =\lambda^{-1} g_{\dot{\gamma}(l)} \big( \dot{\gamma}(l),\dot{c}(0) \big), \]
where $\lambda=\max\{1,F(-\dot{c}(0))\}$.
This yields $\pi-\ora{\angle}(pxc(\ve))=\ola{\angle}(pxc(-\ve))$,
so that we shall set
$\omega:=\pi-\ora{\angle}(pxc(\ve))=\ola{\angle}(pxc(-\ve))$ in the following discussion.
\end{remark}

\begin{lemma}\label{lem2.3.new1}
For each $ t \in [0, l]$, the $g_{\dot{\gamma}}$-orthogonal component $J^{\perp}(t)$
to $\dot{\gamma}(t)$ is given by 
\[
J^{\perp} (t) := J(t) - \frac{g_{\dot{\gamma}(l)} (\dot{\gamma}(l),\dot{c}(0))}{l}t\dot{\gamma}(t).
\]
\end{lemma}

\begin{proof}
Since $J$ is a Jacobi field along $\gamma$, there exist $a,b \in \R$ 
satisfying $g_{\dot{\gamma}(t)} (\dot{\gamma}(t),J(t))= a t + b$ for all $t \in [0, l]$. 
Since $J(0) = 0$, we find $b = 0$. 
Together with $J (l) = \dot{c}(0)$, we obtain
$a l = g_{\dot{\gamma}(l)} ( \dot{\gamma}(l),\dot{c} (0))$
and hence
\[
g_{\dot{\gamma}(t)} \big( \dot{\gamma}(t),J(t) \big)
= \frac{g_{\dot{\gamma}(l)} ( \dot{\gamma}(l),\dot{c}(0))}{l} t.
\]
Thus the Jacobi field $J^{\perp}$ as above
is $g_{\dot{\gamma}}$-orthogonal to $\dot{\gamma}(t)$ on $[0, l]$. 
$\qedd$
\end{proof}

The {\em index form} with respect to $\gamma|_{[0, \,l]}$ is defined by 
\[
I_{l} (X, Y) 
:= 
\int^{l}_{0} 
\left\{ 
g_{\dot{\gamma}} (D_{\dot{\gamma}}^{\dot{\gamma}}X, 
D_{\dot{\gamma}}^{\dot{\gamma}}Y)-
g_{\dot{\gamma}} (
R^{\dot{\gamma}}(X, \dot{\gamma}) \dot{\gamma}, Y
)
\right\}
dt
\]
for $C^{\infty}$-vector fields $X, Y$ along $\gamma|_{[0, \,l]}$. 
Recall \eqref{def_T_curv} for the definition of the tangent curvature $\cT_M$.

\begin{lemma}\label{lem2.3}
Set $L(s):=d(p,c(s))$.
Then we have
\[
L'(0) = g_{\dot{\gamma}(l)} \big(  \dot{\gamma}(l),\dot{c}(0) \big), \qquad
L''(0) = I_{l} (J^{\perp}, J^{\perp}) -\cT_{M} \big( \dot{\gamma}(l), \dot{c}(0) \big).
\]
\end{lemma}

\begin{proof}
These are consequences of the fundamental first and second variational formulas
(cf.\ \cite[Exercise 5.1.4]{BCS}, \cite[Exercise 5.2.7]{BCS}).
We only remark that the geodesic equation $D_{\dot{c}}^{\dot{c}} \dot{c} \equiv 0$ implies
$\cT_{M} (\dot{\gamma}(l), \dot{c}(0)) 
= - g_{\dot{\gamma}(l)} (D_{\dot{c}(0)}^{\dot{\gamma}(l)} \dot{c}(0), \dot{\gamma}(l))$.
$\qedd$
\end{proof}

\subsection{Key lemma}\label{sec3.2}

Throughout this subsection, we assume that the radial flag curvature of $(M,F,p)$
is bounded below by that of a von Mangoldt surface $(\wt{M},\tilde{p})$ as in Theorem~\ref{TCT}.
Note that 
\[
f'(t) < 0 \quad \text{on} \ (\rho, \infty),
\]
since $f'(\rho) =0$ and $G(\rho)\ne0$ for a unique $\rho \in (0, \infty)$. 
Given small $\delta>0$, we modify $(\wt{M},\tilde{p})$ into $(\wt{M}_{\delta}, \tilde{o})$
with the metric $d\tilde{s}^2_{\delta} = dt^2 + f_{\delta}(t)^{2} d\theta^2$
on $(0, \infty) \times \Sph_{\tilde{o}}^{1}$ such that $f_{\delta}$ satisfies 
\[ f''_{\delta} + (G - \delta) f_{\delta}= 0,
 \qquad f_{\delta}(0) = 0, \qquad f_{\delta}'(0) = 1,
\]
where $G$ is the radial curvature function of $\wt{M}$.
Since $(\wt{M}_{\delta}, \tilde{o})$ has the less curvature than $(\wt{M},\tilde{p})$,
we can also employ $(\wt{M}_{\delta}, \tilde{o})$ as a reference surface for $M$.
Note that $(\wt{M}_{\delta}, \tilde{o})$ is again of von Mangoldt type satisfying 
\begin{equation}\label{2013_09_16_rem1}
f'_{\delta}(\rho_{\delta}) = 0 \quad \text{and} \quad G(\rho_{\delta}) - \delta \ne0
\end{equation}
for some unique $\rho_{\delta} \in (\rho, \infty)$,
and that $\lim_{\delta \downarrow 0} \rho_{\delta} = \rho$ 
as well as $\lim_{\delta \downarrow 0} f_{\delta} (t) = f(t)$.

Let $c$, $x=c(0)$, $\gamma$ and $l=d(p,x)$ be as in the previous subsection.
Fix a point $\tilde{x} \in \wt{M}_{\delta}$ with $ \tilde{d}_\delta(\tilde{o}, \tilde{x})= l$,
where $\tilde{d}_\delta$ denotes the distance function of $d\tilde{s}^2_{\delta}$. 
Let $\tilde{\gamma} : [0,l] \lra \wt{M}_{\delta}$ be the minimal geodesic segment
from $\tilde{o}$ to $\tilde{x}$, and take a unit parallel vector field $\wt{E}$
along $\tilde{\gamma}$ orthogonal to $\dot{\tilde{\gamma}}$.
Define the Jacobi field $\wt{X}$ along $\tilde{\gamma}$ by 
\[
\wt{X}(t) := \frac{1}{f_{\delta}(l)} f_{\delta}(t) \wt{E}(t),
\]
and denote by $\wt{I}_{l}(\,\cdot\,,\,\cdot\,)$ the index form of $\wt{M}_{\delta}$ 
for $C^{\infty}$-vector fields along $\tilde{\gamma}|_{[0,\,l]}$. 

\begin{lemma}\label{lem2.5.new1}
For any Jacobi field $X$ along $\gamma$ which is $g_{\dot{\gamma}}$-orthogonal to
$\dot{\gamma}$ and satisfies $X(0) = 0$ and $g_{\dot{\gamma}(l)} (X (l), X(l)) = 1$, we have
\[
\wt{I}_{l} (\wt{X}, \wt{X}) 
\ge I_{l}(X, X)
+ \frac{\delta}{f_{\delta}(l)^{2}} \int_{0}^{l} f_{\delta}(t)^{2}\,dt.
\]
\end{lemma}

\begin{proof}
Let $E$ be the vector field along $\gamma$ such that $E(l) = X(l)$ and
$E':=D_{\dot{\gamma}}^{\dot{\gamma}}E \equiv 0$.
Put
\[
Y(t) := \frac{1}{f_{\delta}(l)} f_{\delta} (t) E(t)
\]
and note that $X(0) = Y(0)$ and $X(l) =Y(l)$.
Thus, since $X$ is a Jacobi field, it follows from the basic index lemma
(cf.~\cite[Lemma 7.3.2]{BCS}) that $I_l(Y,Y) \ge I_l(X,X)$.
Combining this with the radial curvature bound $K_M(\dot{\gamma}(t),Y(t)) \ge G(t)$
(recall \eqref{eq:flag} for the definition of the flag curvature $K_M$), we obtain
\begin{align*}
\wt{I}_{l} (\wt{X}, \wt{X}) 
&= \int^{l}_{0} 
\{\langle \wt{X}', \wt{X}' \rangle - (G - \delta) \|\wt{X}\|^{2} \} \,dt\\
&= \int^{l}_{0} \left\{ g_{\dot{\gamma}} (Y', Y')
 -(G - \delta) g_{\dot{\gamma}} (Y, Y) \right\} dt\\
&\ge \int^{l}_{0} \left\{ g_{\dot{\gamma}} (Y', Y')
 -K_M (\dot{\gamma}, Y) g_{\dot{\gamma}} (Y, Y) \right\} dt
 + \delta \int^{l}_{0} g_{\dot{\gamma}} (Y, Y) \,dt\\
&= I_{l} (Y, Y) + \frac{\delta}{f_{\delta}(l)^{2}} \int_{0}^{l}f_{\delta}(t)^{2}\,dt\\
&\ge I_{l} (X, X) + \frac{\delta}{f_{\delta}(l)^{2}} \int_{0}^{l}f_{\delta}(t)^{2}\,dt.
\end{align*}
$\qedd$
\end{proof}

Fix a geodesic $\tilde{c}:(-\ve,\ve) \lra \wt{M}_{\delta}$
with $\tilde{c}(0)=\tilde{x}$ such that
\begin{equation}\label{eq:c}
\angle\big( \dot{\tilde{\gamma}}(l),\dot{\tilde{c}}(0) \big) =\omega,
 \qquad \|\dot{\tilde{c}}\|=\lambda:=\max\left\{ 1,F\!\left( -\dot{c}(0) \big) \right. \right\},
\end{equation}
where $\omega$ is as in Remark~\ref{rem_2012_01_21_3.1}.
(There may be two choices for such $\tilde{c}$, whereas there is no difference between them
since $\wt{M}_{\delta}$ is a surface of revolution.)
This choice of $\tilde{c}$ is the key trick in dealing with the non-reversible case.
Let us consider the geodesic variation
\[ \tilde{\varphi}(t,s):=\exp_{\tilde{o}}\left( \frac{t}{l}\exp^{-1}_{\tilde{o}}\big( \tilde{c}(s) \big) \right),
 \qquad (t,s) \in [0,l] \times (-\ve,\ve). \]
By setting 
\[
\wt{J}(t) := \frac{\partial \wt{\varphi}}{\partial s} (t, 0),
\]
we get the Jacobi field $\wt{J}$ 
along $\tilde{\gamma}$ with $\wt{J}(0) = 0$ and $\wt{J}(l) = \dot{\tilde{c}}(0)$. 
Similarly to Lemma~\ref{lem2.3.new1}, the Jacobi field 
\[
\wt{J}^{\perp} (t) 
:= 
\wt{J}(t) 
- 
\frac{\langle 
\dot{\tilde{\gamma}}(l), \dot{\tilde{c}}(0)
\rangle}{l} t \dot{\tilde{\gamma}}(t)
\]
along $\tilde{\gamma}$ is orthogonal to $\dot{\tilde{\gamma}}(t)$ on $[0,l]$. 

\begin{lemma}\label{lem2.5}
Assume that $B^+_{\rho}(p) \cap B^-_{2r}(q) = \emptyset$ and 
$g_{\dot{\gamma}(l)}(v,v)  \ge F(v)^2$ for all $v \in T_xM$.
Then there exists $\delta_{1}= \delta_{1}(f,r)> 0$ such that, for any $\delta \in (0, \delta_{1})$, 
we have
\[
\wt{I}_{l} (\wt{J}^{\perp}, \wt{J}^{\perp}) 
-I_{l} (J^{\perp}, J^{\perp}) 
\ge 
\delta C_1 \sin^{2} \omega,
\]
where $J$ is as in Subsection~$\ref{sec3.1}$ and we set for later convenience
\[ C_1:=\frac{1}{2 f(l_{0})^{2}} \int_{0}^{l_{0}} f(t)^{2} \,dt, \qquad l_0:=d(p,q).  \]
\end{lemma}

\begin{proof}
First of all, $x \in B^-_r(q)$ and $B^+_{\rho}(p) \cap B^-_{2r}(q) = \emptyset$ ensure
$l=d(p,x) \ge \rho+r$.
Thus $l>\rho_{\delta}$ holds for $\delta \in (0,\delta_1)$ with sufficiently small 
$\delta_1=\delta_1(f,r)>0$.

We shall show
\begin{equation}\label{lem2.5_2012_02_05_inequ}
\wt{I}_{l} (\wt{J}^{\perp}, \wt{J}^{\perp}) 
- I_{l} (J^{\perp}, J^{\perp}) 
\ge
\frac{\delta \sin^{2} \omega}{f_{\delta}(l)^{2}} \int_{0}^{l}f_{\delta}(t)^{2}\,dt.
\end{equation}
Recall from Remark~\ref{rem_2012_01_21_3.1}
that $g_{\dot{\gamma}(l)}(\dot{\gamma}(l),\dot{c}(0))=\lambda \cos\omega$.
From \eqref{eq:c}, we observe
\[ \wt{J}^{\perp} (l) = \dot{\tilde{c}}(0) 
-\lambda \cos \omega \cdot \dot{\tilde{\gamma}}(l) = \pm \lambda \sin \omega \cdot \wt{E}(l) 
= \pm \lambda \sin \omega \cdot \wt{X}(l), \]
where $\wt{E}$ and $\wt{X}$ are as in the paragraph preceding Lemma~\ref{lem2.5.new1}.
Since both $\wt{J}^{\perp}$ and $\wt{X}$ are Jacobi fields (on $\wt{M}_{\delta}$), this implies
$\wt{J}^{\perp}(t) = \pm \lambda \sin \omega \cdot \wt{X}(t)$ on $[0, l]$. 
Hence we obtain
\begin{equation}\label{lem2.5-2}
\wt{I}_{l} (\wt{J}^{\perp}, \wt{J}^{\perp}) 
= (\lambda \sin\omega)^2 \wt{I}_{l} (\wt{X}, \wt{X}).
\end{equation}
By Lemma~\ref{lem2.3.new1}, we have
\[ g_{\dot{\gamma}(l)} \big( J^{\perp}(l),J^{\perp}(l) \big) = 
g_{\dot{\gamma}(l)} \big( \dot{c} (0), \dot{c} (0) \big) -(\lambda \cos \omega)^2. \]
By the hypothesis
$g_{\dot{\gamma}(l)}(\dot{c}(0),\dot{c}(0)) \ge \max\{ F(\dot{c}(0))^2, F(-\dot{c}(0))^2 \}=\lambda^2$,
\[
a:=g_{\dot{\gamma}(l)} \big( J^{\perp}(l),J^{\perp}(l) \big) -(\lambda \sin \omega)^2 \ge 0
\]
holds. Thus Lemma~\ref{lem2.5.new1} shows
\[ \wt{I}_l(\wt{X},\wt{X}) \ge \frac{I_l(J^{\perp},J^{\perp})}{a+(\lambda \sin \omega)^2}
 +\frac{\delta}{f_{\delta}(l)^2} \int_0^l f_{\delta}(t)^2 \,dt. \]
Combining this with \eqref{lem2.5-2}, we obtain
\begin{align*}
\wt{I}_l(\wt{J}^{\perp},\wt{J}^{\perp}) -I_l(J^{\perp},J^{\perp})
&\ge -a\wt{I}_l(\wt{X},\wt{X})
 +\frac{\delta \{ a+(\lambda \sin \omega)^2 \}}{f_{\delta}(l)^2} \int_0^l f_{\delta}(t)^2 \,dt \\
&\ge -a\wt{I}_l(\wt{X},\wt{X})
 +\frac{\delta \sin^2 \omega}{f_{\delta}(l)^2} \int_0^l f_{\delta}(t)^2 \,dt.
\end{align*}
Since $l>\rho_{\delta}$, $f'_\delta (l) < 0$ holds from \eqref{2013_09_16_rem1}. 
Hence we get 
\[
\wt{I}_l(\wt{X},\wt{X})
=\frac{1}{f_{\delta}(l)^2} \int_0^l \{ (f'_{\delta})^2 +f_{\delta} f''_{\delta} \} \,dt
=\frac{f'_{\delta}(l)}{f_{\delta}(l)}<0,
\]
which completes the proof of \eqref{lem2.5_2012_02_05_inequ}.

Since $|l-l_0| \le \max\{d(q,x),d(x,q)\}<r$ and $l,l_0 >\rho$,
taking smaller $\delta_1(f,r)>0$ if necessary, we have
\[ \frac{1}{f_{\delta}(l)^{2}} \int_{0}^{l}f_{\delta}(t)^{2}\,dt 
> 
\frac{1}{2f(l_{0})^{2}} \int_{0}^{l_{0}}f(t)^{2}\,dt \]
for all $\delta \in (0, \delta_{1})$. 
$\qedd$
\end{proof}

Put $L(s):=d(p,c(s))$ as in Lemma~\ref{lem2.3}, as well as
$\wt{L}(s):=\tilde{d}_{\delta}(\tilde{o},\tilde{c}(s))$.

\begin{lemma}[Key lemma]\label{lem2.6} 
In addition to the assumptions in Lemma~$\ref{lem2.5}$,
we assume that $\cT_M(\dot{\gamma}(l), \dot{c}(0)) = 0$.
Then, for each $\delta \in (0, \delta_{1})$ and $\theta \in (0, \pi/2)$,  
there exists $\ve' =\ve'(M,l,f,\ve,\delta,\theta) \in (0, \ve)$ such that
$L(s) \le \wt{L}(s)$ holds for all $s \in [-\ve', \ve']$,
provided that $\sin \omega \ge \sin \theta$.
\end{lemma}

\begin{proof} 
Set
\[ \cR(s) := L(s) - \left\{ L(0) + L'(0)s + \frac{1}{2} L''(0)s^{2} \right\}. \]
By virtue of $\cT_M(\dot{\gamma}(l), \dot{c}(0)) = 0$,
it follows from Lemma~\ref{lem2.3} and Remark~\ref{rem_2012_01_21_3.1} that  
\[ L(s) =l+s \lambda \cos \omega + \frac{s^{2}}{2} I_{l} (J^{\perp}, J^{\perp}) +\cR (s). \]
Since $\cR(s)=O(|s|^3)$ and $c$ lives in a bounded domain,
there exists $C_2=C_2(M,l)>0$ such that
$|\cR(s)| \le C_2|s|^3$ for all $s \in (-\ve,\ve)$.
This yields
\begin{equation}\label{lem2.6-10} 
L(s) \le l+s \lambda \cos \omega + \frac{s^{2}}{2} I_{l} (J^{\perp}, J^{\perp})
+ C_2 |s|^3
\end{equation}
for $s \in (-\ve,\ve)$.
Now, by the same argument on $\wt{M}_{\delta}$, we obtain
\[ \wt{L}(s) \ge l+s \lambda \cos \omega + \frac{s^{2}}{2} 
\wt{I}_{l} (\wt{J}^{\perp}, \wt{J}^{\perp}) -C_{3} |s|^{3} \]
for some $C_3=C_3(f,l)>0$ and all $s \in (-\ve,\ve)$.
Combining this with \eqref{lem2.6-10} and Lemma~\ref{lem2.5}, we have
\begin{align*}
\wt{L}(s) - L(s) 
&\ge
\frac{s^{2}}{2} 
\left\{
\wt{I}_{l} (\wt{J}^{\perp}, \wt{J}^{\perp}) 
- I_{l} (J^{\perp}, J^{\perp})
\right\} - (C_2 + C_3)|s|^{3}\\
&\ge
\frac{\delta C_{1} \sin^{2} \omega}{2} s^{2}
- (C_2 + C_3)|s|^{3}.
\end{align*}
Since $\sin\omega \ge \sin\theta$ by hypothesis, we further observe
\[ \wt{L}(s) - L(s) \ge \frac{s^2}{2}\{ \delta C_1 \sin^2 \theta -2(C_2+C_3)s \}. \]
Therefore choosing
\[ \ve' := \min \left\{ \ve, \frac{\delta C_{1} \sin^2 \theta}{2(C_2+C_3)} \right\} \]
shows $L(s) \le \wt{L}(s)$ for all $s \in [-\ve',\ve']$.
$\qedd$
\end{proof}

See Remark~\ref{rm:zero} below for the case of $\sin\omega=0$.
We remark that, because $\lim_{\delta \to 0}\ve'=0$,
one can not take the limit as $\delta \to 0$ (i.e., $\wt{M}_{\delta} \to \wt{M}$) at this stage.

\section{Weak TCT for thin triangles}\label{sec4}

From Lemma~\ref{lem2.6}, we immediately derive TCT for thin triangles outside the cut locus of $p$,
with respect to the model surface $(\wt{M}_{\delta},\tilde{o})$.
Let $(M,F,p)$, $(\wt{M},\tilde{p})$ and $(\wt{M}_{\delta},\tilde{o})$
be as in Subsection~\ref{sec3.2}.
We remark that the sector 
\[ \wt{V}_\delta (\pi) : = \{ \tilde{x} \in \wt{M}_\delta  \, | \, 0 \le \theta(\tilde{x}) < \pi \} \]
has no pair of cut points (see Remark~\ref{rem5.1_2012_02_07}).
Take $\ve'' \in (0,\ve']$ such that
$\tilde{c}([-\ve'',\ve'']) \subset \wt{V}_{\delta}(\pi)$ always holds
for $\tilde{c}$ as in Subsection~\ref{sec3.2} (by rotating $\wt{M}_{\delta}$ if necessary).

\begin{lemma}[Weak TCT]\label{2011_07_08_lem3.10}
In the same situation as Lemma~$\ref{lem2.6}$,
we further assume that
\begin{enumerate}[{\rm (1)}]
\item
$g_v(w,w) \ge F(w)^2$ for all $s \in [-\ve'',\ve'']$, $v \in \cG_p(c(s))$ and $w \in T_{c(s)}M$,
\item
$\cT_M(v,\dot{c}(s))=0$ for any $s \in [-\ve'',\ve'']$ and $v \in \cG_p(c(s))$,
and $\bar{c}(s):=c(\ve''-s)$ $(s \in [-\ve'',\ve''])$ is geodesic.
\end{enumerate}
By choosing smaller $\ve''$ $($depending on $M$ and $l)$ if necessary,
we can also assume that $\bar{c}$ is minimal.
Take $a \in (0,\ve'']$ and put $y:=c(a)$.
Then the forward triangle $\triangle (\ora{px}, \ora{py}) \subset M$ admits
a comparison triangle $\triangle (\tilde{o}\tilde{x} \tilde{y})$ in $\wt{M}_{\delta}$
such that $\ora{\angle} x \ge \angle \tilde{x}$ and $\ola{\angle} y \ge \angle \tilde{y}$.
\end{lemma}

\begin{proof}
Since $\bar{c}$ is geodesic and minimal,
$F(-\dot{c})$ is constant and $L_{\rm{m}}(c)=d_{\rm{m}}(x,y)=\lambda d(x,y)$.
Putting $V:=\nabla d(p,\cdot)$ on $M \setminus \Cut(p)$
and denoting by $d_V$ the distance function with respect to $g_V$, we have
\begin{equation}\label{eq:tri}
L_{\rm{m}}(c)=d_{\rm{m}}(x,y) \le d_V(x,y) \le d_V(x,p)+d_V(p,y) =d(p,x)+d(p,y).
\end{equation}
Thanks to $d(x,y)=a \le \ve''$ and the triangle inequality \eqref{eq:tri},
we can take a comparison triangle $\triangle (\tilde{o}\tilde{x} \tilde{y})$
with $0=\theta(\tilde{x}) \le \theta(\tilde{y}) <\pi$.
Comparing $\tilde{d}_{\delta}(\tilde{o},\tilde{y})=d(p,y)$ with $L(a) \le \wt{L}(a)$ in Lemma~\ref{lem2.6},
we find $\angle \tilde{x} \le \pi-\omega=\ora{\angle} x$.
We obtain $\angle \tilde{y} \le \ola{\angle} y$ by a similar argument
(via $L(-a) \le \wt{L}(-a)$ in Lemma~\ref{lem2.6}, beginning from $y$ instead of $x$).
$\qedd$
\end{proof}

\begin{remark}\label{rm:zero}
We excluded the case of $\sin\omega=0$ (i.e., $g_{\dot{\gamma}(l)}(\dot{\gamma}(l),\dot{c}(0))=\pm \lambda$)
in Lemma~\ref{2011_07_08_lem3.10} without loss of generality.
If $\sin\omega=0$, then we observe from the estimates in \eqref{eq:angle} that
$\dot{\gamma}(l)=\dot{c}(0)$ with $F(\dot{c}(0)) \ge F(-\dot{c}(0))$
($\ora{\angle}x=\pi$, $\ola{\angle}y=0$) or
$\dot{\gamma}(l)=-\dot{c}(0)/F(-\dot{c}(0))$ with $F(\dot{c}(0)) \le F(-\dot{c}(0))$
($\ora{\angle}x=0$, $\ola{\angle}y=\pi$).
Since $\bar{c}$ is also geodesic, $c$ is contained in a geodesic passing through $p$ in both cases.
Thus TCT clearly holds.
\end{remark}

\begin{remark}\label{rm:conj}
Lemma~\ref{2011_07_08_lem3.10} holds true also in the case of
$B^-_{2r}(q) \cap (\Conj(p) \cup \{p\}) =\emptyset$,
where $\Conj(p)$ denotes the conjugate locus of $p$.
We put $\omega=\pi-\ora{\angle}(pxc(\ve))$ and choose $\gamma$ from $p$ to $x$
such that $g_{\dot{\gamma}(l)}(\dot{\gamma}(l),\dot{c}(0))=\lambda \cos\omega$.
Consider the geodesic $\tilde{c}:[0,\ve) \lra \wt{M}_{\delta}$ as in \eqref{eq:c}
and compare $\wt{L}(s)$ with $L(s):=\int_0^l F((\partial \varphi/\partial t)(t,s)) \,dt$,
where $t \longmapsto \varphi(t,s)$ is the geodesic segment from $p$ to $c(s)$
such that $(\partial \varphi/\partial t)(0,s)$ is in a neighborhood of $\dot{\gamma}(0)$
on which $\exp_p$ is diffeomorphic.
Then the same argument as Subsection~\ref{sec3.2} shows
$\wt{L}(s) \ge L(s) \ge d(p,c(s))$ for all sufficiently small $s>0$,
and hence the analogue of Lemma~\ref{2011_07_08_lem3.10} holds.
\end{remark}

\section{Double triangle lemma}\label{sec5}

Throughout this section, let $(\wt{M}, \tilde{p})$ be a von Mangoldt surface of revolution.
The following fact on the cut loci of $\wt{M}$ is important.

\begin{remark}\label{rem5.1_2012_02_07}
The cut locus $\Cut (\tilde{x})$ of $\tilde{x} \not= \tilde{p}$ is either an empty set, 
or a ray properly contained in the meridian $\theta^{-1} (\theta (\tilde{x}) + \pi)$ opposite to $\tilde{x}$.
Moreover, the endpoint of $\Cut(\tilde{x})$ is the first conjugate point to $\tilde{x}$ 
along the minimal geodesic from $\tilde{x}$ passing through $\tilde{p}$ (\cite[Main Theorem]{T}).
\end{remark}

We start with simple lemmas.

\begin{lemma}\label{2011_06_24_lem4.1}
Take $\tilde{x} \in \wt{M} \setminus \{\tilde{p}\}$ and 
let $\tilde{\gamma} :[0, \infty) \lra \wt{M}$ be the meridian passing through $\tilde{x}$.
Fix $t>0$ and put $\tilde{y}:=\tilde{\gamma}(t)$.
Then we have $\tilde{d} (\tilde{x},\tilde{y}) < \tilde{d} (\tilde{x},\tilde{z})$
for all $\tilde{z} \in \partial B_t (\tilde{p}) \setminus \{\tilde{y}\}$. 
\end{lemma}

\begin{proof}
Fix arbitrary $\tilde{z} \in \partial B_t (\tilde{p})$.
If $\tilde{d}(\tilde{p},\tilde{x}) \le t$, then we have
\[ \tilde{d}(\tilde{x},\tilde{y}) =t-\tilde{d}(\tilde{p},\tilde{x})
 =\tilde{d}(\tilde{p},\tilde{z}) -\tilde{d}(\tilde{p},\tilde{x})
 \le \tilde{d}(\tilde{x},\tilde{z}) \]
by the triangle inequality, and equality holds only if $\tilde{z}=\tilde{y}$.
In the case where $\tilde{d}(\tilde{p},\tilde{x})>t$, we similarly find
\[ \tilde{d}(\tilde{y},\tilde{x}) =\tilde{d}(\tilde{p},\tilde{x})-t
 =\tilde{d}(\tilde{p},\tilde{x}) -\tilde{d}(\tilde{p},\tilde{z})
 \le \tilde{d}(\tilde{z},\tilde{x}) \]
and equality holds only if $\tilde{z}=\tilde{y}$.
$\qedd$
\end{proof}

\begin{lemma}\label{2011_06_24_lem4.2}
Let $\tilde{x} \in \wt{M} \setminus \{\tilde{p}\}$, $t>0$
and take $\tilde{y}, \tilde{z} \in \partial B_t (\tilde{p})$. 
If $0= \theta(\tilde{x}) < \theta(\tilde{y}) < \theta (\tilde{z}) \le \pi$,
then $\tilde{d} (\tilde{x}, \tilde{y}) <\tilde{d} (\tilde{x}, \tilde{z})$ holds.
\end{lemma} 

\begin{proof}
We can assume $\theta(\tilde{z})<\pi$ without loss of generality.
Let $\tilde{\gamma}:[0, \infty) \lra \wt{M}$ be the meridian passing through 
$\tilde{y}$, and $\wt{xz}$ be the geodesic segment from $\tilde{x}$ to $\tilde{z}$. 
Then $\wt{xz} \cap \tilde{\gamma}([0, \,\infty)) \neq \emptyset$
since $0<\theta(\tilde{y})<\theta(\tilde{z})<\pi$, so that we take
$\tilde{w} \in \wt{xz} \cap \tilde{\gamma}([0, \,\infty))$.
It follows from Lemma~\ref{2011_06_24_lem4.1} that
$\tilde{d} (\tilde{w}, \tilde{y}) <\tilde{d} (\tilde{w}, \tilde{z})$, and hence
\[ \tilde{d} (\tilde{x}, \tilde{y}) \le \tilde{d} (\tilde{x}, \tilde{w}) + \tilde{d} (\tilde{w}, \tilde{y})
< \tilde{d} (\tilde{x}, \tilde{w}) + \tilde{d} (\tilde{w}, \tilde{z}) = \tilde{d} (\tilde{x}, \tilde{z}). \]
$\qedd$
\end{proof}

\begin{lemma}\label{2011_06_24_lem4.3}
Let 
$\triangle(\tilde{p}\tilde{x}\tilde{y})$ and $\triangle(\tilde{p}\tilde{y}\tilde{z})$ 
be geodesic triangles in $\wt{M}$ such that
$0= \theta(\tilde{x}) < \theta(\tilde{y}) < \theta (\tilde{z})$ 
and $\angle (\tilde{p}\tilde{y}\tilde{x}) + \angle (\tilde{p}\tilde{y}\tilde{z}) \neq \pi$. 
If there is a geodesic triangle 
$\triangle(\tilde{p}\tilde{q}\tilde{r})$ in $\wt{M}$ satisfying 
$\tilde{d} (\tilde{p}, \tilde{q}) = \tilde{d} (\tilde{p}, \tilde{x})$, 
$\tilde{d} (\tilde{p}, \tilde{r}) = \tilde{d} (\tilde{p}, \tilde{z})$, and 
$\tilde{d} (\tilde{q}, \tilde{r}) = \tilde{d} (\tilde{x}, \tilde{y}) + \tilde{d} (\tilde{y}, \tilde{z})$, 
then $\theta (\tilde{z}) < \pi$ holds.
\end{lemma}

\begin{proof}
Suppose $\theta (\tilde{z})>\pi$.
Put $t:=\tilde{d} (\tilde{p}, \tilde{z})$ and let $\tilde{v} \in \partial B_t (\tilde{p})$
be the point with $\theta (\tilde{v}) = \pi$. 
Then Lemma~\ref{2011_06_24_lem4.2} for $\tilde{y}$ and
$\tilde{v},\tilde{z} \in \partial B_t(\tilde{p})$ shows
$\tilde{d} (\tilde{y}, \tilde{v}) <\tilde{d} (\tilde{y}, \tilde{z})$.
Together with the triangle inequality, we find
\begin{equation}\label{2011_06_24_lem4.3_2}
\tilde{d} (\tilde{x}, \tilde{v}) 
\le \tilde{d} (\tilde{x}, \tilde{y}) + \tilde{d} (\tilde{y}, \tilde{v})
< \tilde{d} (\tilde{x}, \tilde{y}) + \tilde{d} (\tilde{y}, \tilde{z}).
\end{equation}
Since $\tilde{d} (\tilde{p}, \tilde{q}) = \tilde{d} (\tilde{p}, \tilde{x})$ and
$\tilde{d} (\tilde{p}, \tilde{r}) = \tilde{d} (\tilde{p}, \tilde{z})=t$,
we can take $\tilde{w} \in \partial B_t (\tilde{p})$ satisfying 
$\tilde{d} (\tilde{x}, \tilde{w}) = \tilde{d} (\tilde{q}, \tilde{r})$.
Combining this with \eqref{2011_06_24_lem4.3_2} yields
\begin{equation}\label{2011_06_24_lem4.3_4}
\tilde{d} (\tilde{x}, \tilde{v}) 
< \tilde{d} (\tilde{x}, \tilde{y}) + \tilde{d} (\tilde{y}, \tilde{z}) 
= \tilde{d} (\tilde{q}, \tilde{r}) =\tilde{d} (\tilde{x}, \tilde{w}).
\end{equation}
Note that, however, $\triangle(\tilde{p}\tilde{x}\tilde{w})$ is isometric to 
$\triangle(\tilde{p}\tilde{q}\tilde{r})$.
Since $\pi = \theta (\tilde{v}) \ge \angle( \tilde{q}\tilde{p}\tilde{r}) 
= \angle( \tilde{x}\tilde{p}\tilde{w})$, it follows from Lemma~\ref{2011_06_24_lem4.2} that 
$\tilde{d} (\tilde{x}, \tilde{w}) \le \tilde{d} (\tilde{x}, \tilde{v})$.
This contradicts \eqref{2011_06_24_lem4.3_4}.

If $\theta(\tilde{z})=\pi$, then
$\tilde{d}(\tilde{x},\tilde{z})<\tilde{d} (\tilde{x}, \tilde{y}) + \tilde{d} (\tilde{y}, \tilde{z})
 =\tilde{d} (\tilde{x}, \tilde{w})$ similarly implies a contradiction,
where $\tilde{d}(\tilde{x},\tilde{z})<\tilde{d} (\tilde{x}, \tilde{y}) + \tilde{d} (\tilde{y}, \tilde{z})$
follows from the hypothesis
$\angle (\tilde{p}\tilde{y}\tilde{x}) + \angle (\tilde{p}\tilde{y}\tilde{z}) \neq \pi$.
$\qedd$
\end{proof}

The following lemma is the main result of this section.

\begin{lemma}{\bf (Double triangle lemma)}\label{2011_06_28_lem4.4}
Let 
$\triangle(\tilde{p}\tilde{x}\tilde{y})$, $\triangle(\tilde{p}\tilde{y}\tilde{z})$ 
be geodesic triangles in $\wt{M}$ such that
$0= \theta(\tilde{x}) < \theta(\tilde{y}) < \theta (\tilde{z})$ and 
$\angle (\tilde{p}\tilde{y}\tilde{x}) + \angle (\tilde{p}\tilde{y}\tilde{z}) \le \pi$. 
If there is a geodesic triangle 
$\triangle(\tilde{p}\tilde{q}\tilde{r})$ in $\wt{M}$ satisfying 
$\tilde{d} (\tilde{p}, \tilde{q}) = \tilde{d} (\tilde{p}, \tilde{x})$, 
$\tilde{d} (\tilde{p}, \tilde{r}) = \tilde{d} (\tilde{p}, \tilde{z})$, and 
$\tilde{d} (\tilde{q}, \tilde{r}) = \tilde{d} (\tilde{x}, \tilde{y}) + \tilde{d} (\tilde{y}, \tilde{z})$,
then we have $\angle \tilde{x} \ge \angle \tilde{q}$ and $\angle \tilde{z} \ge \angle \tilde{r}$.
\end{lemma}

\begin{proof}
If $\angle (\tilde{p}\tilde{y}\tilde{x}) + \angle (\tilde{p}\tilde{y}\tilde{z}) = \pi$,
then $\triangle(\tilde{p}\tilde{x}\tilde{z})$ is isometric to $\triangle(\tilde{p}\tilde{q}\tilde{r})$,
and $\angle \tilde{x}=\angle \tilde{q}$ as well as $\angle \tilde{z}=\angle \tilde{r}$ hold.
Assume $\angle (\tilde{p}\tilde{y}\tilde{x}) + \angle (\tilde{p}\tilde{y}\tilde{z}) < \pi$
and put $t:=\tilde{d}(\tilde{p},\tilde{z})$.
Note that
\[ \tilde{d}(\tilde{x},\tilde{z})<\tilde{d} (\tilde{x}, \tilde{y}) + \tilde{d} (\tilde{y}, \tilde{z})
 =\tilde{d}(\tilde{q},\tilde{r}). \]
Since $\theta(\tilde{z})<\pi$ by Lemma~\ref{2011_06_24_lem4.3},
let us consider the continuous curve $\eta:[0,\pi-\theta(\tilde{z})] \lra \partial B_t(\tilde{p})$
with $\eta(0)=\tilde{z}$ and $\theta(\eta(s))=\theta(\tilde{z})+s$.
Take $\tilde{v}=\eta(s_0)$ such that $\tilde{d}(\tilde{x},\tilde{v})=\tilde{d}(\tilde{q},\tilde{r})$,
and observe $\theta(\tilde{z})<\theta(\tilde{v}) \le \pi$ from $\tilde{d}(\tilde{x},\tilde{z}) <\tilde{d}(\tilde{x},\tilde{v})$.
For any $s \in (0,s_0]$, Lemma~\ref{2011_06_24_lem4.2} shows
\[ \tilde{d}\big( \tilde{x},\eta(s) \big) \le \tilde{d}(\tilde{x},\tilde{v})
 =\tilde{d}(\tilde{x},\tilde{y}) +\tilde{d}(\tilde{y},\tilde{z})
 <\tilde{d}(\tilde{x},\tilde{y}) +\tilde{d}\big( \tilde{y},\eta(s) \big), \]
so that the minimal geodesic emanating from $\tilde{x}$ and passing through $\tilde{y}$
does not cross $\eta([0,s_0])$.
Therefore $\angle (\tilde{p}\tilde{x}\tilde{y}) > \angle (\tilde{p}\tilde{x}\tilde{v})$ must hold.
Thus we have
\[ \angle \tilde{x}=\angle (\tilde{p}\tilde{x}\tilde{y}) > \angle (\tilde{p}\tilde{x}\tilde{v})
 =\angle (\tilde{p}\tilde{q}\tilde{r}) =\angle \tilde{q}, \]
and similarly $\angle \tilde{z}>\angle \tilde{r}$.
$\qedd$
\end{proof}

\section{Proof of the main theorem}\label{sec6}

This section is devoted to the proof of our main theorem.
Throughout this section, let $(M,F,p)$, $(\wt{M},\tilde{p})$ and $(\wt{M}_{\delta},\tilde{o})$
be as in Subsection~\ref{sec3.2}.

\begin{definition}[Locality]\label{def6.1_2012_02_10}
A forward triangle $\triangle (\ora{px}, \ora{py}) = (p, x, y; \gamma, \sigma, c)$ in $M$ 
is said to be {\em local} if it admits a comparison triangle in $\wt{M}$ and if
$c((0,d(x,y)])$ contains no cut point of $x$.
\end{definition}

Note that a local triangle admits a comparison triangle also in $\wt{M}_{\delta}$.
We prove a simple lemma for later convenience.

\begin{lemma}\label{lm:loc}
Under the assumptions in Lemma~$\ref{2011_07_08_lem3.10}$,
if the forward triangle $\triangle (\ora{px}, \ora{py}) = (p, x, y; \gamma, \sigma, c) \subset M$ 
admits a comparison triangle in $\wt{M}$, then so does
$\triangle (\ora{px}, \ora{pc(s)})$ for any $s \in (0,a)$, where $a:=d(x,y)$.
\end{lemma}

\begin{proof}
Let $s_1$ be the maximum of $s \in [0,a]$ such that, for every $r \in (0,s)$,
$\triangle (\ora{px}, \ora{pc(r)})$ admits a comparison triangle.
Note that $s_1>0$ thanks to \eqref{eq:tri}.
If $s_1<a$, then $\triangle (\ora{px}, \ora{pc(s_1)})$ admits a comparison triangle
$\triangle(\tilde{o}\tilde{x}\wt{c(s_1)}) \subset \wt{M}$ with
$\angle(\tilde{x}\tilde{o}\wt{c(s_1)})=\pi$.
Take $s_1<s_2< \cdots <s_N=a$ such that $\triangle (\ora{pc(s_i)}, \ora{pc(s_{i+1})})$
admits a comparison triangle $\triangle(\tilde{o}\wt{c(s_i)}\wt{c(s_{i+1})})$
for each $i=1,2,\ldots,N-1$, where
$\theta(\tilde{x})<\theta(\wt{c(s_1)})<\cdots<\theta(\wt{c(s_N)})$.
Then $\theta(\wt{c(s_N)}) -\theta(\tilde{x})>\pi$ and
$L_{\rm{m}}(c) =\tilde{d}(\tilde{x},\wt{c(s_1)})
 +\sum_{i=1}^{N-1} \tilde{d}(\wt{c(s_i)},\wt{c(s_{i+1})})$
show that $\triangle (\ora{px}, \ora{py})$ does not admit a comparison triangle in $\wt{M}$,
this is a contradiction.
$\qedd$
\end{proof}

\begin{lemma}\label{2011_07_08_lem4.5}
If a local triangle 
$\triangle (\ora{px}, \ora{py}) = (p, x, y; \gamma, \sigma, c)$ in $M$ satisfies the conditions
\begin{enumerate}[{\rm (1)}]
\item
$c([0,a]) \subset M \setminus \ol{B^+_{\rho}(p)}$, where $a:=d(x,y)$,
\item
$g_v(w,w) \ge F(w)^2$ for all $s \in [0,a]$, $v \in \cG_p(c(s))$ and $w \in T_{c(s)}M$,
\item
$\cT_M(v,\dot{c}(s)) = 0$ for all $s \in [0, a]$ and $v \in \cG_p(c(s))$,
and the reverse curve of $c$ is geodesic,
\item
$c([0,a]) \cap \CC(p) = \emptyset$, where $\CC(p):=\Cut(p) \cap \Conj(p)$,
\end{enumerate}
then it admits a comparison triangle $\triangle (\tilde{o}\tilde{x} \tilde{y})$ in $\wt{M}_\delta$
such that $\ora{\angle} x \ge \angle \tilde{x}$ and $\ola{\angle} y \ge \angle \tilde{y}$.
\end{lemma}

\begin{proof}
We can exclude the trivial case where $c$ is contained in a geodesic passing through $p$
(recall Remark~\ref{rm:zero}).
Let $S$ be the set of $s \in (0, a)$ such that there is a comparison triangle 
$\triangle (\tilde{o}\tilde{x} \wt{c(s)}) \subset \wt{V}_{\delta}(\pi)$ (with, say, $\theta(\tilde{x})=0$)
of $\triangle (\ora{px}, \ora{pc(s)}) \subset M$ satisfying 
$\ora{\angle} x \ge \angle \tilde{x}$ and $\ola{\angle}(pc(s)x) \ge \angle \wt{c(s)}$. 
It is sufficient to prove $\sup S = a$.
Lemma~\ref{2011_07_08_lem3.10} and Remark~\ref{rm:conj} ensure that $S$ is non-empty.

Now, suppose $\sup S < a$ and take $s_0 \in S$ with $s_0+\ve'' >\sup S$
for $\ve''$ chosen in Section~\ref{sec4}.
Then there is a comparison triangle 
$\triangle (\tilde{o}\tilde{x}_1 \tilde{y}_1) \subset \wt{M}_{\delta}$ 
of $\triangle (\ora{px}, \ora{pc(s_{0})})$ such that 
\begin{equation}\label{2011_08_01_lem4.5_1}
\ora{\angle} x \ge \angle \tilde{x}_1, \qquad \ola{\angle} \big( pc(s_0)x \big) \ge \angle \tilde{y}_1.
\end{equation}
By Lemma~\ref{2011_07_08_lem3.10} and Remark~\ref{rm:conj},
for $\zeta \in (\sup S -s_0,\ve'']$,
$\triangle (\ora{pc(s_0)}, \ora{pc(s_{0} + \zeta)})$ admits a comparison triangle
$\triangle (\tilde{o}\tilde{x}_2 \tilde{y}_2) \subset \wt{M}_{\delta}$ such that
\begin{equation}\label{2011_08_01_lem4.5_2}
\ora{\angle} \big( pc(s_0)c(s_0+\zeta) \big) \ge \angle \tilde{x}_2, \qquad
 \ola{\angle} \big( pc(s_0 + \zeta)c(s_0) \big) \ge \angle \tilde{y}_2.
\end{equation}
We can take $\tilde{x}_2=\tilde{y}_1$ with 
$0 = \theta (\tilde{x}_1) < \theta (\tilde{y}_1) = \theta (\tilde{x}_2) < \theta (\tilde{y}_2)$. 
Note that Theorem~\ref{2011_05_27_thm1} shows
\[ \ola{\angle}\big( pc(s_0)x \big) + \ora{\angle} \big(pc(s_0)c(s_0+\zeta) \big) \le \pi, \]
so that $\angle \tilde{y}_1 + \angle \tilde{x}_2 \le \pi$.
Indeed, putting $\lambda=\max\{ 1,F(-\dot{c}(s_0)) \}$
(and recalling the proof of Lemma~\ref{2011_05_27_lem1}), we have
\begin{align*}
&\cos \ola{\angle}\big( pc(s_0)x \big) +\cos \ora{\angle}\big( pc(s_0)c(s_0+\zeta) \big) \\
&= \lambda^{-1} \max\left\{ g_v\big( v,\dot{c}(s_0) \big) \,\big|\, v \in \cG_p\big( c(s_0) \big) \right\}
 -\lambda^{-1} \min\left\{ g_v\big( v,\dot{c}(s_0) \big) \,\big|\, v \in \cG_p\big( c(s_0) \big) \right\} \\
&\ge 0.
\end{align*}
Hence it follows from Lemmas~\ref{2011_06_28_lem4.4}, \ref{lm:loc} that
there exists a comparison triangle $\triangle (\tilde{o}\tilde{x}_3 \tilde{y}_3) \subset \wt{V}_\delta(\pi)$
of $\triangle(\ora{px},\ora{pc(s_0+\zeta)})$ such that $\angle \tilde{x}_3 \le \angle \tilde{x}_1$
and $\angle \tilde{y}_3 \le \angle \tilde{y}_2$.
Combining these with \eqref{2011_08_01_lem4.5_1} and \eqref{2011_08_01_lem4.5_2}
implies $s_{0} + \zeta \in S$, which contradicts $s_0+\zeta >\sup S$.
$\qedd$
\end{proof}

To remove the hypothesis $c([0,a]) \cap \CC(p) =\emptyset$
and get $\wt{M}_{\delta}$ back to $\wt{M}$, we need the following lemmas.

\begin{lemma}{\rm (see \cite[Lemma 2]{IT1})}\label{lem2.1}
The Hausdorff dimension of $\CC(p)$ is at most $n - 2$. 
\end{lemma}

\begin{lemma}\label{lem2.2}
Assume that $x \not\in \Conj (p)$, $y \not\in \Cut (x)$, and $p \not\in c ([0, d(x, y)])$,
where $c$ denotes the unit speed minimal geodesic segment emanating from $x$ to $y$. 
Then, for each $v \in T_yM \cap F^{-1}(1)$, there exists a sequence 
$\{ c_{i}:[0,a_{i}] \lra M \}_{i \in \N}$
of unit speed minimal geodesic segments with $c_{i}(0)=x$ 
convergent to $c$ such that $c_{i}([0,a_{i}]) \cap \CC(p) = \emptyset$ and 
\[
\lim_{i \to \infty} \frac{1}{F(\exp_{y}^{-1} (c_{i}(a_{i})) )} \exp_{y}^{-1}\big( c_{i}(a_{i}) \big) = v.
\]
\end{lemma}

\begin{proof}
This follows from Lemma~\ref{lem2.1} and a similar argument to \cite[Lemma 3.5]{KT3}.
$\qedd$
\end{proof}

\begin{lemma}\label{lem6.8_2012_02_11}
Assume that a local triangle $\triangle (\ora{px}, \ora{py}) = (p, x, y; \gamma, \sigma, c) \subset M$
satisfies $c([0,d(x,y)]) \subset M \setminus \ol{B^+_{\rho}(p)}$ and that,
for some neighborhood $\cN(c) \subset M \setminus \ol{B^+_{\rho}(p)}$ of $c$,
\begin{enumerate}[{\rm (1)}]
\item 
$g_v(w, w)  \ge F(w)^2$ for all $z \in \cN(c)$, $v \in \cG_p(z)$ and $w \in T_zM$,
\item
$\cT_{M}(v,w) = 0$ for all $z \in \cN(c)$, $v \in \cG_p(z)$ and $w \in T_zM$,
and the reverse curve of $c$ is geodesic. 
\end{enumerate}
Then $\triangle (\ora{px}, \ora{py})$ admits a comparison triangle $\triangle (\tilde{p}\tilde{x} \tilde{y})$ 
in $\wt{M}$ with $\ora{\angle} x \ge \angle \tilde{x}$ and $\ola{\angle} y \ge \angle \tilde{y}$.
\end{lemma}

\begin{proof}
It suffices to construct a comparison triangle $\triangle(\tilde{o}\tilde{x}\tilde{y})$ in $\wt{M}_{\delta}$
satisfying the angle conditions, and then take the limit as $\delta \to 0$.
Put $a:=d(x,y)$.

We first assume $x \not\in \Conj(p)$.
By the definition of the locality, $y$ is not a cut point of $x$.
Thus, by Lemma~\ref{lem2.2}, 
there exists a sequence $\{ c_{i}:[0, a_{i}] \lra M \}_{i \in \N}$
of unit speed minimal geodesic segments with $c_{i}(0)=x$ convergent to $c$ such that 
$c_{i}([0, a_{i}]) \cap \CC(p) = \emptyset$.
Then, for sufficiently large $i$, the forward triangle $\triangle (\ora{px}, \ora{pc_i(a_i)})$ 
is local and $c_i([0,a_i]) \subset \cN(c)$.
Hence Lemma~\ref{2011_07_08_lem4.5} guarantees that $\triangle (\ora{px}, \ora{pc_i(a_i)})$ 
admits a comparison triangle 
$\triangle(\tilde{o}\tilde{x}\wt{c_{i}(a_{i})}) \subset \wt{M}_{\delta}$ such that
\[ \ora{\angle} \big( pxc_i(a_i) \big) \ge \angle \tilde{x}, \qquad
 \ola{\angle}\big( pc_i(a_i)x \big) \ge \angle \wt{c_i(a_i)}. \]
Since $\lim_{i \to \infty} \dot{c}_{i}(0) = \dot{c}(0)$ and
$\lim_{i \to \infty} \dot{c}_{i}(a_i) = \dot{c}(a)$, we find
\[ \ora{\angle} x =\lim_{i \to \infty} \ora{\angle}\big( pxc_i(a_i) \big) \ge \angle \tilde{x}, \qquad 
 \ola{\angle} y =\lim_{i \to \infty} \ola{\angle}\big( pc_i(a_i)x \big)
 \ge \lim_{i \to \infty} \angle \wt{c_i(a_i)}. \]
Therefore $\triangle (\tilde{o}\tilde{x}\tilde{y}) 
 := \lim_{i \to \infty} \triangle(\tilde{o}\tilde{x}\wt{c_i(a_i)})$ 
is the desired comparison triangle.\par 
Next we consider the case of $x \in \Conj(p)$.
Take sufficiently small $\ve > 0$ such that $y\not\in\Cut (x_\ve)$, 
where $x_\ve := \gamma(d(p, x) - \ve)$. 
Let $c_{\ve}:[0,a_{\ve}] \lra M$ be the unit speed minimal geodesic segment
emanating from $x_{\ve}$ to $y$.
We can assume $c_{\ve}([0,a_{\ve}]) \subset \cN(c)$, because 
$\lim_{\ve \to 0}c_{\ve}=c$ by the locality. 
Since $x_\ve \not\in \Conj (p)$, we can apply the argument above to
$\triangle (\ora{px_\ve}, \ora{py})$, 
and obtain its comparison triangle
$\triangle (\tilde{o} \tilde{x}_\ve\tilde{y})$ 
with $\ora{\angle}(px_{\ve}y) \ge \angle \tilde{x}_{\ve}$
and $\ola{\angle}(pyx_{\ve}) \ge \angle \tilde{y}$.
Then $\triangle (\tilde{o}\tilde{x}\tilde{y}) 
 := \lim_{\ve \to 0} \triangle (\tilde{o} \tilde{x}_\ve\tilde{y})$
is the desired comparison triangle.
$\qedd$
\end{proof}

\noindent {\it Proof of Theorem~$\ref{TCT}$.}
We can prove the theorem by applying Lemma~\ref{lem6.8_2012_02_11} to thin triangles
and by the same argument as the proof of Lemma~\ref{2011_07_08_lem4.5}.
$\qedd$

\medskip

\begin{flushleft}
K.~Kondo, M.~Tanaka\\[1mm]
Department of Mathematics, Tokai University,\\
Hiratsuka City, Kanagawa Pref. 259-1292, Japan\\[0.5mm]
{\small e-mail: {\tt keikondo@keyaki.cc.u-tokai.ac.jp},
{\tt tanaka@tokai-u.jp}}
\bigskip

S.~Ohta\\[1mm]
Department of Mathematics, Kyoto University,\\
Kyoto 606-8502, Japan\\[0.5mm]
{\small e-mail: {\tt sohta@math.kyoto-u.ac.jp}}
\end{flushleft}

\end{document}